\documentclass{llncs}
\usepackage{amssymb,amsfonts,latexsym,stmaryrd}

\usepackage[PostScript=dvips]{diagrams}
\usepackage{proof}
\usepackage{QED}
\usepackage{epsfig}

\newcommand{\beqa}{\begin{eqnarray*}}
\newcommand{\eeqa}{\end{eqnarray*}\par\noindent}

\newcommand{\id}{\mathsf{id}}

\newcommand{\PP}{\mathbf{P}}
\newcommand{\pow}{\mathcal{P}}
\newcommand{\VV}{\mathcal{V}}
\newcommand{\SSS}{\mathcal{S}}

\newcommand{\rarr}{\rightarrow}

\newcommand{\ie}{\textit{i.e.}~}
\newcommand{\CC}{\mathcal{C}}

\newcommand{\Set}{\mathbf{Set}}
\newcommand{\Two}{\mathbf{2}}
\newcommand{\Term}{\mathbf{1}}
\newcommand{\Sub}{\mathsf{Sub}}

\newcommand{\And}{\; \wedge \;}

\newcommand{\lsem}{\llbracket}
\newcommand{\rsem}{\rrbracket}
\newcommand{\Ass}{\boxplus}
\newcommand{\preord}{\lesssim}
\newcommand{\riso}{\rTo^{\cong}}
\newcommand{\FCoalg}{F{-}\mathbf{Coalg}}
\newcommand{\pkap}{\pow_{\kappa}}
\newcommand{\vn}{\varnothing}

\newarrow{Eq}=====
\newarrow{mon}>--->
\newarrow{rel}--+->
\newarrow{dash}....>

\title{From Lawvere to Brandenburger-Keisler:\\
 interactive forms of diagonalization and self-reference}
\author{Samson Abramsky \inst{1} and Jonathan Zvesper \inst{1}}
\institute{Oxford University Computing Laboratory}

\begin{document}

\maketitle

\section{Introduction}

Diagonal arguments lie at the root of many fundamental phenomena in the foundations of logic and mathematics. Recently, a striking form of diagonal argument has appeared  in the foundations of epistemic game theory, in a paper by Adam Brandenburger and H. Jerome Keisler \cite{BK}. The core Brandenburger-Keisler result can be seen, as they observe, as a two-person or interactive version of Russell's Paradox. This  raises a number of fascinating questions at the interface of epistemic game theory, logic and theoretical computer science:
\begin{enumerate}
\item Is the Brandenburger-Keisler argument (henceforth: `BK argument') just one example of a more general phenomenon, whereby mathematical structures and arguments can be generalized from a familiar `one-person' form to a two- or multi-agent interactive form?
\item To address this question, a sharper understanding of the BK argument is needed.
The argument hinges on a statement involving the modalities \textbf{believes} and \textbf{assumes}. 
The statement has the form
\begin{center}
Ann \textbf{believes} that Bob \textbf{assumes} that \ldots
\end{center}
which is not familiar as it stands. Where does this \textbf{believes-assumes} pattern come from?
How exactly does it relate to the more familiar arguments in the one-person case? In particular, can it be reduced to a one-person argument?
\item Is there a natural multi-agent generalization of the BK argument? In particular, does it have a \emph{compositional structure}, which allows a smooth generalization to any number of agents?
\item The main formal consequence of the BK argument is that there can be no belief model  which is `assumption-complete' with respect to a collection of predicates including those definable in the first-order language of the model. Brandenburger and Keisler also give a positive result, a construction of a topological  model which is assumption-complete with respect to the positive fragment of first-order logic extended with the \textbf{believes} and \textbf{assumes} modalities.
They raise the question of a more general perspective on the availability of such models.
\end{enumerate}

We shall provide substantial answers to questions (2)--(4) above in the present paper.
These results also suggest that the Brandenburger-Keisler `paradox' does offer a good point of entry for considering the more general question (1).

The starting point for our approach is a classic paper by F. William Lawvere from 1969 \cite{Law}, in which he gave a simple form of the (one-person) diagonal argument as a \emph{fixpoint lemma}  in a very general setting. This lemma lies at the basis of a remarkable range of results. Lawvere's ideas were amplified and given a very attractive presentation in a recent paper by Noson Yanofsky \cite{Yan}.

Our contributions can be summarized as follows:
\begin{itemize}
\item We reformulate the core BK argument as a \emph{fixpoint lemma}. This immediately puts it in the general genre of diagonal arguments, and in particular of the Lawvere fixpoint lemma.
\item The BK argument applies to (belief) \emph{relations}, while the Lawvere argument applies to \emph{functions} (actually, abstractly to arrows in a category). To put them on common ground, we give a novel relational reformulation of the Lawvere argument.
\item We analyze the exact logical resources required for our fixpoint version of the BK argument, and show that it can be carried out in \emph{regular logic}, the fragment of first-order logic comprising sequents $\phi \vdash \psi$, where $\phi$ and $\psi$ are built from atomic formulas by conjunction and existential quantification. Regular logic can be interpreted in any \emph{regular category}, which covers a wide range of types of mathematical structure. The Lawvere argument can also be carried out in (a fragment of) regular logic. We can now recognize the Lawvere argument as exactly the one-person version of the BK argument, and interpret the key BK lemma as a reduction to the one-person Lawvere argument.
\item This analysis leads in turn to a smooth generalization of the BK argument to multi-agent belief models. The content of the \textbf{believes-assumes} pattern, or more generally the $\mbox{\textbf{believes}}^*$-\textbf{assumes} pattern:
\begin{center}
$A_1$ \textbf{believes} that $A_2$ \textbf{believes} that  \ldots $A_n$ \textbf{believes} that  $B$ \textbf{assumes} that \ldots
\end{center}
is that the Lawvere hypothesis of \emph{weak point surjectivity} is propagated back along \emph{belief chains}.
\item We furthermore give a \emph{compositional analysis} of the \textbf{believes-assumes} pattern, which \emph{characterizes} what we call `belief-complete' relations in terms of this propagation property. This gives a rather definitive analysis for why the BK argument takes the form it does.
\item We then turn to the issue of the construction of assumption complete models. The categorical perspective allows us to apply general techniques from coalgebra and domain theory to the construction of such models. 
\item Finally, we should mention the use of (elementary) methods from category theory and categorical logic in this context as a methodological contribution. While these may be unfamiliar to some, we believe that they are fully justified in allowing the full scope and generality of the results, and the mathematical contexts in which they may be applied, to be exposed.
\end{itemize}

The further contents of this paper are as follows. In Section~2, we review the setting for the BK argument, and give our formulation of it as a fixpoint lemma. In Section~3, we show how it can be formalized in regular logic. In Section~4, we introduce the Lawvere fixpoint lemma.
In Section~5 we bring BK and Lawvere together, giving a relational reformulation of the Lawvere lemma, and showing how to reduce BK to this version of Lawvere, \ie the two-person to the one-person argument. In Section~6, we give the multi-agent generalization, and in Section~7 the compositional analysis of belief-completeness, and hence of the \textbf{believes-assumes} pattern. In Section~8, we show how general functorial methods lead to the construction of assumption-complete models. Section~9 concludes with some further directions.

\section{The Brandenburger-Keisler Argument}
A (two-person) \emph{belief structure} has the form $(U_a, U_b, R_a, R_b)$ where
\[ R_a \subseteq U_a \times U_b, \qquad R_b \subseteq U_b \times U_a . \]
In the context of epistemic game theory, we think of  $U_a$ and $U_b$ as \emph{type spaces} for Alice and Bob:
\begin{itemize}
\item Elements of $U_a$  represent possible epistemic states of Alice in which she holds beliefs about Bob, Bob's beliefs, etc.
Symmetrically, elements of $U_b$ represent possible epistemic states of Bob.

\item The relations $R_a \subseteq U_a \times U_b$, $R_b \subseteq U_b \times U_a$ specify these beliefs.
Thus $R_a(x, y)$ expresses that in state $x$, Alice believes that state $y$ is possible for Bob.

\item We say that a state $x \in U_a$ \emph{believes} $P \subseteq U_b$ if
$R_a(x) \subseteq P$.
Modal logic provides a useful perspective on these notions, as shown by Eric Pacuit \cite{Pacuit07} (see also \cite{BK}).
Modally, `$x$ believes $P$' is just $x \models \Box_a P$ where $\Box_{a}$ is the usual necessity operator defined with respect to the relation $R_{a}$:
\[  x \models \Box_a \phi \; \equiv \; \forall y. \, R_a(x, y) \Rightarrow \; y \models \phi . \]
\item We say that $x$ \emph{assumes} $P$ if 
$R_a(x) = P$.
This is $x \models \Ass_a P$, where $\Ass_a$ is the modality defined by
\[  x \models \Ass_a \phi \; \equiv \; \forall y. \, R_a(x, y) \Leftrightarrow \; y \models \phi . \]

\end{itemize}

A belief structure $(U_a, U_b, R_a, R_b)$  is \emph{assumption-complete} \cite{BK} with respect to a collection of predicates on $U_a$ and $U_b$ if for every predicate $P$ on $U_b$ in the collection, there is a state $x$ on $U_a$ such that $x$ assumes $P$; and similarly for the predicates on $U_a$.  (A \emph{predicate} on a set $U$ is just a subset of $U$.)\footnote{
Related forms of completeness assumption are used in the analysis 
of various solution concepts in games in \cite{BS,BFK}.}

Brandenburger and Keisler show in \cite{BK} that this hypothesis, in the case where the predicates include those definable in the first-order language of this structure, leads to a contradiction. (They also show the existence of assumption complete models for some other cases.)

Our aim is to understand the general structures underlying this argument. Our first step is to recast their result as a \emph{positive} one --- a fixpoint lemma.

\subsection{The BK Fixpoint Lemma}
We are given a belief structure $(U_a, U_b, R_a, R_b)$.
We assume that for `all'  (in some `definable' class of) predicates $p$ on $U_a$ there is $x_0$ such that:

\begin{eqnarray}
\label{eqn:BAp}
R_a (x_0) \subseteq \{ y \mid R_b (y) & = & \{ x \mid p(x) \} \}. \\
\label{eqn:DT}
\exists y. \, R_a(x_0, y).
\end{eqnarray}

Modally, these assumptions can be expressed as follows: 
\[ x_0 \models \Box_a \boxplus_b p  \And \Diamond_a \top . \]

\textbf{Remark} We can read (\ref{eqn:BAp}) as saying: `$x_0$ \textbf{believes} that ($y$ \textbf{assumes} that $p$)', in the terminology of Brandenburger and Keisler. 

\begin{lemma}[Basic Lemma]
From (\ref{eqn:BAp}) and (\ref{eqn:DT}) we have:
\[ p(x_0) \;\; \Longleftrightarrow \;\; \exists y. [R_a(x_0, y) \; \wedge \; R_b(y, x_0)] . \]
\end{lemma}
\begin{proof}
Suppose $p(x_{0})$. Take $y$ as in (\ref{eqn:DT}), so $R_a(x_0, y)$. Then by (\ref{eqn:BAp}), $R_b(y, x_0)$. 
Now consider $y$ satisfying $R_a(x_0, y) \; \wedge \; R_b(y, x_0)$. By (\ref{eqn:BAp}), from $R_a(x_0, y)$ we have that $R_b (y) = \{ x \mid p(x) \}$. Hence from $R_b(y, x_0)$ we have that $p(x_{0})$.
\end{proof}

\begin{lemma}[BK Fixpoint Lemma]
Under our assumptions, every unary propositional operator $O$ has a fixpoint.
\end{lemma}

\begin{proof}
Since $p$ was arbitrary, we can define
\begin{eqnarray}
\label{eqn:q}
q(x) & \equiv & \exists y. [R_a(x, y) \; \wedge \; R_b(y, x)] \\
\label{eqn:p}
p(x) & \equiv & O(q(x)).
\end{eqnarray}
(N.B.~It is important that $p$ is defined \emph{without reference to $x_0$} to avoid circularity.)  
These definitions combined with the equivalence given by the Basic Lemma immediately yield:
\[ O(q(x_{0})) \stackrel{(\ref{eqn:p})}{\equiv} p(x_0) \; \Longleftrightarrow \; \exists y. [R_a(x_0, y) \; \wedge \; R_b(y, x_0) \stackrel{(\ref{eqn:q})}{\equiv} q(x_0), \]
so $q(x_0)$ is a fixpoint for the operator $O$, as required.
\end{proof}

\paragraph{Remarks}
Taking $O \equiv \neg$ yields the BK `paradox'.
(In fact $\neg q(x)$ is equivalent to their `diagonal formula' $D$ in \cite{BK}).

In general, since our assumptions (\ref{eqn:BAp}) and (\ref{eqn:DT}) are relative to a class of predicates, this argument relies on $q(x)$ and $p(x)$ being in this class. Note that $q(x)$ only involves conjunction and existential quantification. This leads to our analysis of the logical resources needed to carry out the BK argument.

\section{Formalizing BK in Regular Logic}
We recall that \emph{regular logic} is the fragment of (many-sorted) first-order logic comprising sequents of the form
\[ \phi \vdash_X \psi \]
where $\phi$ and $\psi$ are built from atomic formulas by conjunction (including the empty conjunction $\top$) and existential quantification; and $X$ is a finite set of variables which includes all those occurring free in $\phi$ and $\psi$. The intended meaning of such a sequent is
\[ \forall x_1 \cdots \forall x_n [ \phi \Rightarrow \psi] \]
where $X = \{ x_1 , \ldots , x_n \}$.
This is a common fragment of intuitionistic and classical logic. It plays a core r\^ole in categorical logic.
A convenient summary of regular logic can be found in the lecture notes by Carsten Butz \cite{Butz}.

We shall write $\vdash_X \psi$ for the sequent $\top \vdash_X \psi$, and $\phi \vdash \psi$ for $\phi \vdash_{\varnothing} \psi$.

We shall assume a logical vocabulary containing the sorts $U_a$ and $U_b$, and binary relation symbols $R_a : U_a \times U_b$ and $R_b : U_b \times U_a$, together with a constant $c : U_a$ which will correspond to $x_0$ in the informal argument given in the previous section. Thus $c$ is associated with the given predicate $p$, which will be represented by a formula in one free variable of sort $U_a$.

The assumptions given in the informal argument can be expressed as regular sequents as follows.
\[ \begin{array}{ll}
(A1) & R_a(c, y) \And R_b(y, x) \vdash_{\{x, y\}} p(x) \\
(A2) & R_a(c, y) \And p(x) \vdash_{\{x, y\}} R_b(y, x) \\
(A3) & \vdash \exists y. \, R_a(c, y)
\end{array}
\]
Here (A1) and (A2) correspond to assumption (\ref{eqn:BAp}) in the informal argument, while (A3) corresponds to assumption (\ref{eqn:DT}).

The formal version of Lemma~1 is as follows:
\begin{lemma}
From (A1)--(A3) we can infer the following sequents:
\[ \begin{array}{ll}
(F1) & p(c) \vdash q(c) \\
(F2) & q(c) \vdash p(c)
\end{array}
\]
where 
\[ q(x) \; \equiv \; \exists y. [R_a(x, y) \; \wedge \; R_b(y, x)]. \]
\end{lemma}
A \emph{definable unary propositional operator} will be represented by a formula context $O[\cdot]$, which is a closed formula built from atomic formulas, plus a `hole' $[\cdot]$. We obtain a formula $O[\phi]$ by replacing every occurrence of the hole by a formula $\phi$.

The formal version of the Fixpoint Lemma is now stated as follows:
\begin{lemma}
Under the assumptions (A1)--(A3), every definable unary propositional operator $O[\cdot]$ has a fixpoint, \ie a sentence $S$ such that
\[ S \vdash O[S], \qquad O[S] \vdash S . \]
\end{lemma}
This is obtained directly from the previous lemma, taking $p(x) \equiv O[q(x)]$.
The required sentence $S$ is then $q(c)$.

\paragraph{Remarks}

\begin{itemize}
\item 
Regular logic can be interpreted in any \emph{regular category} \cite{JvO,Butz}: well-powered with finite limits and images, which are stable under pullbacks.\footnote{There is a brief review of these notions in Section~\ref{catrevsec}.}
These are exactly the categories which support a good calculus of relations.

\item The BK fixpoint lemma is valid in any such category. Regular categories are abundant --- they include all (pre)toposes, all abelian categories,  all equational varieties of algebras, compact Hausdorff spaces, and categories of $Q$-sets for right quantales $Q$.

\item If the propositional operator $O$ is fixpoint-free, the result must be read contrapositively, as showing that the assumptions (A1)--(A3) lead to a contradiction. This will of course be the case if $O = \neg[\cdot]$ in either classical or intuitionistic logic. This yields exactly the BK argument.

\item In other contexts, this need not be the case. For example if the propositions (in categorical terms, the subobjects of the terminal object) form a complete lattice, and $O$ is \emph{monotone}, then by the Tarski-Knaster theorem there will indeed be a fixpoint. This offers a general setting for understanding why \emph{positive logics}, in which all definable propositional operators are monotone,  allow the paradoxes to be circumvented.

\end{itemize}

\section{The Lawvere Fixpoint Lemma}

We start off concretely working in $\Set$. 
Suppose we have a function
\[ g : X \rarr \VV^X \]
or equivalently, by cartesian closure:
\[ \hat{g} : X \times X \rarr \VV \]
Think of $\VV$ as a set of `truth values':  $\VV^X$ is the set of `$\VV$-valued predicates'. Then $g$ is showing how predicates on $X$ can be represented by elements of $X$. In terms of $\hat{g}$: a predicate $p : X \rarr \VV$ is representable by $x \in X$ if for all $y \in X$:
\[ p(y) = \hat{g}(x, y) \]
Note that, if predicates `talk about' $X$, then representable predicates allow $X$ to `talk about itself'.

If $g$ is \emph{surjective}, then \emph{every} predicate on $X$  is representable in $X$.
When can this happen?

\begin{proposition}[Lawvere Fixpoint Lemma]
Suppose that $g : X \rarr \VV^X$ is surjective. Then \emph{every} function $\alpha : \VV \rarr \VV$ has a \emph{fixpoint}: $v \in \VV$ such that $\alpha(v) = v$.
\end{proposition}

\begin{proof}
Define a predicate $p$ by
\[ \begin{diagram}
X \times X && \rTo^{\hat{g}} && \VV \\
\uTo^{\Delta} && && \dTo_{\alpha} \\
X && \rTo_{p} && \VV
\end{diagram}
\]
There is $x \in X$ which represents $p$: then
\[ p(x) = \alpha(\hat{g}(\Delta(x))) = \alpha(\hat{g}(x, x)) = \alpha(p(x)) \]
so $p(x)$ is a fixpoint of $\alpha$.
\end{proof}

\paragraph{Remarks on the proof}
Note firstly that the proof is constructive.   The crucial idea is that it uses \emph{two descriptions of $p$} --- one from its definition, one from its representation via $\hat{g}$.
And since $x$ represents $p$, $p(x)$ is (indirect) \emph{self-application}.

But does this make sense?
Say that $X$ has the \emph{fixpoint property}  if every endofunction on $X$ has a fixpoint. 
Of course, \emph{no set with more than one element has the fixpoint property!}

\textbf{Basic example:} $\Two = \{ 0, 1 \}$. The negation
\[ \neg 0 = 1, \qquad \neg 1 = 0 \]
does not have a fixpoint.
So the meaning of the theorem in $\Set$ must be taken \emph{contrapositively}: 
\begin{center}
For all sets $X$, $\VV$ where $\VV$ has more than one element, there is no surjective map
\[ X \rarr \VV^X \]
\end{center}

\paragraph{Two Applications}
\begin{description}
\item[Cantor's Theorem] Take $\VV = \Two$. There is no surjective map
$X \rarr \Two^X$
and hence $| \PP(X) | \not\leq | X |$. 

We can apply the fixpoint lemma to any putative such map, with $\alpha = {\neg}$, to get the usual `diagonalization argument'.

\item[Russell's Paradox] Let $\SSS$ be a `universe' (set) of sets. Let 
$\hat{g} : \SSS \times \SSS \rarr \Two$
define the membership relation:
\[ \hat{g}(x, y) \Leftrightarrow y \in x \]
Then there is a predicate which can be defined on $\SSS$, and which is not representable by any element of $\SSS$. 

Such a predicate is given by the standard Russell set, which arises by applying the fixpoint lemma with $\alpha = \neg$.
\end{description}

\subsection{Abstract Version of the Basic Lemma}

Lawvere's argument was in the setting of cartesian (closed) categories. Amazingly, it only needs finite products.\footnote{In fact, even less suffices: just monoidal structure and a `diagonal' satisfying only point naturality and monoidality.}

Let $\CC$ be a category with finite products. The terminal object (empty product) is written as $\Term$. In $\Set$ it is any one-point set.

\begin{definition}[Lawvere] An arrow $f : A \times A \rarr \VV$ is \emph{weakly point surjective} (wps) if for every $p : A \rarr \VV$ there is an $x : \Term \rarr A$ such that, for all $y : \Term \rarr A$:
\[ p \circ y = f \circ \langle x, y \rangle : \Term \rarr \VV \]
\end{definition}
In this case, we say that \emph{$p$ is represented by $x$}.

\begin{proposition}[Abstract Fixpoint Lemma] Let $\CC$ be a category with finite products. If $f : A \times A \rarr \VV$ is weakly point surjective, then every endomorphism $\alpha : \VV \rarr \VV$ has a fixpoint $v : \Term \rarr \VV$ such that $\alpha \circ v = v$.
\end{proposition}

\begin{proof}
Define $p : A \rarr \VV$ by
\[ \begin{diagram}
A \times A && \rTo^{f} && \VV \\
\uTo^{\Delta_{A}} && && \dTo_{\alpha} \\
A && \rTo_{p} && \VV
\end{diagram}
\]

Suppose $p$ is represented by $x : \Term \rarr A$. Then
\[ \begin{array}{lclr}
p \circ x & = & \alpha \circ f \circ \Delta_{A} \circ x & \mbox{def of $p$} \\
& = & \alpha \circ f \circ \langle x, x \rangle & \mbox{diagonal} \\
& = & \alpha \circ p \circ x & \mbox{$x$ represents $p$.}
\end{array}
\]

So $p \circ x$ is a fixpoint of $\alpha$.
\end{proof}

In \cite{Law}, the Fixpoint Lemma is used to derive G\"odel's First Incompleteness Theorem.
Yanofsky's paper covers many more applications: semantic paradozes (Liar, Berry, Richard), the Halting Problem, existence of an oracle $B$ such that $\mathbf{P}^B \neq \mathbf{NP}^B$, Parikh sentences, L\"ob's paradox, the Recursion theorem, Rice's theorem, von Neumann's self-reproducing automata, \ldots

All of these are `one-person' results. The question of applying this argument to a two-person scenario such as the BK paradox has remained open.

\section{Reducing BK to Lawvere}

How do we relate Lawvere to BK? As we have seen, the BK argument is valid in any regular category. This is pretty general. Nevertheless, BK needs a richer setting than Lawvere. To find common ground between them, we reformulate Lawvere, replacing \emph{maps} by \emph{relations}.

As a preliminary, we firstly review how regular logic is interpreted in any regular category.

\subsection{Brief Review of Categorical Logic}
\label{catrevsec}
We shall assume familiarity with some very basic category theory: the notions of category, functor, natural transformation, epis, monos, isomorphisms, products and pullbacks. Any introductory text, such as the excellent (and gentle) \cite{Pierce}, covers these in a few pages.

We shall briefly review how formulas of regular logic are interpreted in any category with suitable structure --- the regular categories. For a very clear and detailed expository account of this material, see the lecture notes \cite{Butz}. Another excellent set of lecture notes \cite{JvO} covers both the basic category theory and regular logic.

We firstly recall the notion of \emph{subobject} of an object $A$ in a category $\CC$. If $m_1 : M_1 \rmon A$ and $m_2 : M_2 \rmon A$ are monomorphisms, we write $m_1 \preord m_2$ if $m_1$ factors through $m_2$:
\begin{diagram}
M_2 & \rmon^{m_2} & A \\
\udash & \rumon_{m_1} & \\
M_1 & &
\end{diagram}
Note that if such an arrow exists, it is unique, and a monomorphism. In $\Set$, if $m_1$ and $m_2$ are inclusion mappings, the relation expresses that the subset $M_1$ is included in the subset $M_2$.
The relation $\preord$ is a preorder (reflexive and transitive) and we can factor through by the corresponding equivalence relation to form a partial order. If the collection of equivalence classes forms a set (not a proper class) for every object $A$, we say that $\CC$ is \emph{well-powered}.\footnote{There is a minor technicality lurking here. The equivalence classes of monomorphisms may themselves be proper classes. So more precisely, we should ask that this collection is  \emph{in bijection with a set}.} In this case we write $\Sub(A)$ for the set of equivalence classes --- the subobjects of $A$. Now suppose that $\CC$ has finite limits. In a pullback diagram
\begin{diagram}[heads=vee,width=5em]
\bullet\SEpbk & \rTo & \bullet \\
\dmon^{m'}  & & \dmon_{m} \\
A& \rTo_{f} & B
\end{diagram}
it is always true that if $m$ is mono, so is $m'$.  Moreover, this action of $f$ on monos by pullback is monotone with respect to the preorder $\preord$. Hence there is a well defined map $f^* : \Sub(B) \rarr \Sub(A)$. (In the case of $\Set$, pullbacks of monos correspond to inverse images of subsets.) This assignment $f \mapsto f^*$ is moreover (contravariantly) functorial, and we get a functor
\[ \Sub : \CC^{\mathsf{op}} \rarr \Set \]
which assigns $\Sub(A)$ to each object $A$ of $\CC$ and $f^*$ to each arrow $f$.

We interpret a many-sorted logical vocabulary in a regular category $\CC$ by assigning an object of $\CC$ to each sort\footnote{We shall not distinguish notationally between a syntactic sort and the corresponding object.}, an arrow $c : \Term \rarr A$ to each constant $c$ of sort $A$, and a subobject in $\Sub(A_1 \times \cdots \times A_n)$ to each relation symbol $R: (A_1, \ldots , A_n)$.

Substitution is captured by pullback. Two examples will suffice. Given a predicate $P \rmon A$ and a constant $c : \Term \rarr A$,
\[ \lsem P(c) \rsem = c^*(P) . \]
Given a relation $R \rmon A \times A$:
\[ \lsem R(x, x) \rsem = \Delta_A^*(R) . \]
Conjunction is interpreted by pullbacks. The greatest lower bound of subobjects $[m_1]$, $[m_2]$ in the partial order $\Sub(A)$ is computed on representatives by the pullback
\begin{diagram}[heads=vee,width=5em]
\bullet\SEpbk & \rTo & \bullet \\
\dmon  & & \dmon_{m_2} \\
A& \rTo_{m_1} & B
\end{diagram}
If $m_1$ and $m_2$ are inclusion mappings in $\Set$, one can check that the pullback is given by the intersection of the corresponding subsets.

Finally, regular categories allow \emph{existential quantification} to be interpreted. Given a formula $\phi(x, y)$ where $\lsem \phi(x, y) \rsem \in \Sub(A \times B)$, the projection $\pi : A \times B \rarr B$ yields
\[ \pi^* : \Sub(B) \rarr \Sub(A \times B) . \]
In a regular category, this map has a left adjoint
\[ \exists_{\pi} :  \Sub(A \times B) \rarr \Sub(B) \]
which allows existential quantification to be interpreted:
\[ \lsem \exists x.\, \phi(x, y) \rsem = \exists_{\pi}(\lsem \phi(x, y) \rsem ) . \]
Thus formulas of regular logic can be interpreted in regular categories. Now suppose we are given a regular sequent $\phi \vdash_X \psi$, where $X = x_1 : A_1 , \ldots , x_n : A_n$. The interpretations of the formulas $\phi$, $\psi$ live in the same poset of subobjects: $\lsem \phi \rsem, \lsem \psi \rsem \in \Sub(A_1 \times \cdots \times A_n)$. (This is why it is important to specify $X$). Then the sequent is true in the interpretation if $\lsem \phi \rsem \leq \lsem \psi \rsem$.

The rules of regular logic (just standard rules for this fragment of first-order logic) are sound in any regular category, and thus \emph{we can use logic to reason about relations in a wide variety of mathematical contexts}. There is also a form of strong completeness theorem. For further details, see \cite{Butz}.

\subsection{Relational Reformulation of Lawvere}
As a first step, we reformulate Lawvere's notion of \emph{weak point surjectivity} in relational terms.

To see how to do this, imagine the Lawvere wps situation 
\[ \hat{g} : X \times X \rarr \Omega \]
is happening in a \emph{topos}, and $\Omega$ is the subobject classifier. In the case of $\Set$, $\Omega$ is just $\Two$, and we are appealing to the familiar identification $\pow(X) = \Two^X$ of subsets with characteristic functions.

Then this map $\hat{g}$ corresponds to a \emph{relation}
\[ R \rmon X \times X \]
Such a relation is \emph{weakly point surjective} (wps) if for every subobject $p \rmon X$ there is $x : \Term \rarr X$ such that, for all $y : \Term \rarr A$:
\[ \lsem R(x, y) \rsem = \lsem p(y) \rsem  \]
or in logical terms
\[ R(x, y) \Longleftrightarrow p(y) . \]

In fact, a weaker notion suffices to prove the Fixpoint Lemma (cf. \cite{SV}).
We say that $R$ is \emph{very weakly point surjective} (vwps) if for every subobject $p \rmon X$ there is $x : \Term \rarr X$ such that:
\[ \lsem R(x, x) \rsem = \lsem p(x) \rsem . \]

\subsection{What is a `propositional operator'?}

To find the right `objective' --- \ie language independent --- notion, once again we consider the topos case, and translate out of that into something which makes sense much more widely.

In a topos, a propositional operator is an endomorphism of the subobject classifier
\[ \alpha : \Omega \rarr \Omega \]
(In more familiar terms: an operator on the lattice of truth values, as  e.g.~in Boolean Algebras with Operators.) This  corresponds to the endomorphism of $\VV$ in Lawvere's original formulation.

Note that by Yoneda,  since
$\Sub \cong \CC(-, \Omega)$,
such  endomorphisms of $\Omega$ correspond bijectively with  \emph{endomorphisms of the subobject functor} --- \ie natural transformations
\[ \tau : \Sub \Longrightarrow \Sub . \]
Thus this is the right semantic notion of `propositional operator' in general. 
Naturality corresponds to \emph{commuting with substitution}.

\subsection{The Relational Lawvere Lemma}

\begin{lemma}[Relational Lawvere fixpoint lemma]
If $R$ is a vwps relation on $X$ in a regular category\footnote{In fact, it suffices to assume that the category is well-powered and has finite limits.}, then every endomorphism of the subobject functor
\[ \tau : \Sub \Longrightarrow \Sub \]
has a  fixpoint.
\end{lemma}
Note that a fixpoint $\mathbf{K} \Term \Longrightarrow \Sub$ from the constant functor valued at the terminal object is determined by its value at $\Sub(\Term)$.

\begin{proof}
We define a predicate $P(x) \equiv \tau(R(x, x))$, so $\lsem P \rsem = \tau_X(\Delta_X^*(R))$.
By vwps, there is $c : \Term \rarr X$ such that:
\[  \lsem P(c) \rsem = c^*(\lsem P \rsem) = \langle c, c \rangle^*(R) = \lsem R(c, c) \rsem . \]

Then
\[ \begin{array}{ll}
\lsem P(c) \rsem & = c^*(\lsem P \rsem) = c^*(\tau_X(\Delta_X^*(R)) = \tau_{\Term}(c^* \circ  \Delta_X^*(R)) \\
&  = \tau_{\Term}((\Delta_X \circ c)^*(R)) = \tau_{\Term}(\langle c, c \rangle^*(R))  \\
& = \tau_{\Term}(c^*(\lsem P \rsem)) = \tau_{\Term}(\lsem P(c) \rsem) . 
\end{array}
\]
\end{proof}

\subsection{From BK to Lawvere}

Now given relations
\[ R_a \rmon A \times B, \qquad R_b \rmon B \times A \]
we can form their relational composition $R \rmon A \times A$:
\[ \lsem R(x_1, x_2) \rsem \equiv \lsem \exists y. \, [ R_a(x_1, y) \And R_b(y, x_2) ] \rsem \]

Our Basic Lemma can now be restated as follows:
\begin{lemma}
If $R_a$ and $R_b$ satisfy the BK assumptions (A1)--(A3), then $R$ is vwps.
\end{lemma}

Hence the relational Lawvere fixpoint lemma applies!
As an immediate Corollary, we obtain:

\begin{lemma}[BK Fixpoint Lemma]
If $R_a$ and $R_b$ satisfy the BK assumptions (A1)--(A3), then every endomorphism of the subobject functor has a fixpoint.
\end{lemma}

\section{Multi-Agent Generalization of BK}
A \emph{multiagent belief structure} in a regular category is
\[ (\{ A_i \}_{i \in I}, \{ R_{ij} \}_{(i, j) \in I \times I} ) \]
where
\[ R_{ij} \rmon A_i \times A_j . \]

A \emph{belief cycle} in such a structure is
\[\begin{diagram}
A & \rrel^{R_1} & A_1 & \rrel^{R_2} & \cdots & \rrel^{R_n} & A_n & \rrel^{R_{n+1}} & A \\
\end{diagram}
\]
where we write $R : B \rrel C$ if $R$ is a relation of the indicated type, \ie a subobject of $B \times C$.

We now formulate  \emph{Generalized BK Assumptions} for such a belief cycle:

For each subobject $p \rmon A$, there is some $c : \Term \rarr A$ such that
\[ \begin{array}{ll}
c \models &  \Box_1 \cdots \Box_n \boxplus_{n+1} p \\
& \wedge \\
& \Diamond_1 \top  \And \Box_1 \Diamond_2 \top \And \cdots \And \Box_1 \cdots \Box_{n-1} \Diamond_n \top
\end{array}
\]
These assumptions can be written straightforwardly as regular sequents.

\paragraph{Multiagent BK Fixpoint Lemma}

We can define the relation $R = R_1 ; \cdots ; R_{n+1} : A \rrel A$.

\begin{lemma}[Generalized Basic Lemma]
Under the Generalized BK assumptions, $R$ is vwps. 
\end{lemma}

Hence the Relational Fixpoint Lemma applies.
Note that in the one-person case $n = 0$, \emph{assumption completeness coincides with weak point surjectivity}.

In modal terms:
\[ c \models \boxplus p \; \equiv \; \forall x. \, R(c, x) \Leftrightarrow p(x) . \]
One-person BK   \emph{is}  (relational) Lawvere!
The force of the BK argument is that the (very) wps property propagates back along \emph{belief chains}.

In particular, this produces the `\textbf{believes-assumes}' construction of BK, or the generalized version
$\mbox{\textbf{believes}}^*$-\textbf{assumes}, 
in which `believes' is iterated $n$ times followed by an `assumes'.

\section{Compositional Analysis}

We shall briefly consider the issue of \emph{compositional gluing} of belief relations with given completeness properties.
For simplicity, we shall conduct our discussion concretely, in terms of relations on sets.
To incorporate the idea of relativization to a set of predicates, we shall assume that each set $A$ is given together with a set $\PP(A) \subseteq \pow(A) \setminus \{ \varnothing \}$ of (non-empty) predicates on $A$.

Suppose we are given a relation $R : A \rrel B$.  We say that $R$ is \emph{assumption-complete} (with respect to $\PP$) if for every $p \in \PP(B)$, for some $x \in A$, for all $y \in B$:
\[ R(x, y)   \; \Leftrightarrow \; p(y) .  \]
This is just wps again, of course.

We say that it is \emph{belief-complete} if for all $y \in B$:
\[ R(x, y)   \; \Rightarrow \; p(y) .  \]
and also $\exists y. R(x,y)$.
Modally, this corresponds to
\[ x \models \Box p \And \Diamond \top . \]
Now suppose we have relations 
\[ R_{ab} : A \rrel B, \qquad R_{bc} : B \rrel C .  \]
We define
\[ \boxplus_{bc} p = \{ y \in B \mid R_{bc}(y) = p \} . \]

\begin{lemma}[Composition Lemma]
\label{complemm}
Suppose that:

\begin{enumerate}
\item $R_{ab}$ is belief-complete with respect to $\PP(B)$.
\item $R_{bc}$ is assumption-complete with respect to $\PP(C)$.
\item For each $p \in \PP(C)$, $\boxplus_{bc} p \in \PP(B)$.
\end{enumerate}
Then the composition
$R_{ac} = R_{ab} ; R_{bc} : A \rrel C$ is assumption-complete with respect to $\PP(C)$.
\end{lemma}

\noindent Note the need for the \emph{comprehension assumption} (3).

We now prove a kind of converse to the Composition Lemma, which \emph{characterises} belief-completeness, and shows \emph{why} the BK assumptions and the \textbf{believes-assumes} pattern arise in this context.

\begin{theorem}[Compositional Characterization]
A relation $R : A \rrel B$ is belief complete with respect to $\PP(B)$ if and only if, for every $S : B \rrel C$ such that
\begin{enumerate}
\item $S$ is assumption complete with respect to $\PP(C)$
\item  $\boxplus_S p \in \PP(B)$ for every $p \in \PP(C)$
\end{enumerate}
the composition $R ; S : A \rrel C$ is assumption complete with respect to $\PP(C)$.
\end{theorem}
\begin{proof}
The left to right implication is Lemma~\ref{complemm}.

For the converse, we suppose that $R$ is not belief-complete for some $p \in \PP(B)$. We let $C = \{ 0, 1\}$, and define $S$ to be the characteristic function of $p$. We take $\PP(C) = \{ q \}$, where $q = \{ 1 \}$. Note that $\boxplus q = p$, and that $S$ is assumption complete with respect to $\PP(C)$ --- indeed, any element of $p$, which by our general assumption on predicates is non-empty, assumes $q$. 

We claim that $R ; S$ is not assumption complete for $q$. Indeed, for any $x \in A$, if $R(x) = \vn$, then $R ; S(x) = \vn$, and so  $x$ does not assume $q$. The only other possibility, since by assumption $R$ is not belief complete with respect to $p$, is that for some $y \not\in p$, $R(x, y)$. In this case, $R;S(x, 0)$, and so $x$ does not assume $q$.
\end{proof}

\paragraph{Remark}
The proof of the Compositional Characterization Theorem assumes that we have the freedom to choose any collection of predicates we like on a given set. It would be useful to have a more general formulation and result.

\section{Functorial Constructions of Assumption-Complete Models}

We now turn to the question of constructing belief models which are assumption complete with respect to a natural class of predicates. The categorical perspective is well-suited to this task. Indeed, leaving aside model-theoretic subtleties, we can identify the problem as essentially one of finding fixpoints for certain `powerset-like' functors. This `recursion in the large' at the level of types, to support `recursion  in the small' at the level of programs, is a familiar theme in Theoretical Computer Science \cite{AJ}. If we think of recursion as enabling self-reference, in formulas rather than programs, we see the link to the ideas being considered here. Powerful general methods are available for finding such fixpoints, as solutions of domain equations \cite{AJ} or final coalgebras \cite{Rut}.

The problem can be phrased as follows, in the setting of the \emph{strategy-based belief models} of \cite{BK}.
We are given strategy sets $S_{a}$, $S_{b}$ for Alice and Bob respectively. We want to find sets of types $T_{a}$ and $T_{b}$ such that
\begin{equation}
\label{domeq}
T_{a} \cong \PP(U_{b}), \qquad T_{b} \cong \PP(U_{a}) 
\end{equation}
where $U_{a} = S_{a} \times T_{a}$ and $U_{b} = S_{b} \times T_{b}$ are the sets of states for Alice and Bob. Naively, $\PP$ is powerset, but in fact it must be a restricted set of subsets (extensions of predicates) defined in some more subtle way, or such a structure would be impossible by mere cardinality considerations. 

Thus a state for Alice is a pair $(s, t)$ where $s$ is a strategy from her strategy-set and $t$ is a type.
Given an isomorphism $\alpha : T_{a} \riso \PP(U_{b})$, we can define a relation $R_{a} : U_{a} \rrel U_{b}$ by:
\[ R_{a}((s, t), (s', t')) \; \equiv \; (s',t') \in \alpha(t) . \]
Note that $(s, t)$ \textbf{assumes} $\alpha(t)$.
Because $\alpha$ is an isomorphism, the belief model $(U_{a}, U_{b}, R_{a}, R_{b})$ is automatically assumption complete with respect to $\PP(U_{a})$ and $\PP(U_{b})$.

In fact, having isomorphisms $\alpha : T_{a} \riso \PP(U_{b})$, $\beta : T_{b} \riso \PP(U_{a})$ is more than is strictly required for assumption completeness. It would be sufficient to have retractions
\[ T_{a} \rhd \PP(U_{b}), \qquad T_{b} \rhd \PP(U_{a}) \]
\ie maps 
\[ r_{a} : T_{a} \rarr \PP(U_{b}), \qquad s_{a}  : \PP(U_{b}) \rarr T_{a} \]
such that $r_{a} \circ s_{a} = \id_{\PP(U_{b})}$, and similarly for $T_{b}$ and $\PP(U_{a})$.\footnote{Brandenburger and Keisler ask only for surjections, but they are working in a setting where surjections can always be split.} However, we shall emphasize the situation where we do have isomorphisms, where we can really speak of \emph{canonical solutions}.

We shall now generalize this situation so as to clarify what the mathematical form of the problem is.
Suppose that we have a category $\CC$, which we assume to have finite products, and a functor $\PP : \CC \rarr \CC$. We are given objects $S_{a}$ and $S_{b}$ in $\CC$. Hence we can define functors $F_{a}, F_{b} : \CC \rarr \CC$:
\[ F_{a}(Y) = \PP(S_{b} \times Y), \qquad F_{b}(X) =  \PP(S_{a} \times X) . \]
Intuitively, $F_{a}$ provides one level of beliefs which Alice may hold about states which combine strategies for Bob with `types' from the `parameter space' $Y$; and symmetrically for $F_{b}$.

Now we define a functor $F : \CC \times \CC \rarr  \CC \times \CC$ on the product category:
\[ F(X, Y) = (F_{a}(Y), F_{b}(X)) . \]
To ask for a pair of isomorphisms as in (\ref{domeq}) is to ask for a \emph{fixpoint} of the functor $F$: an object of $\CC \times \CC$ (hence a pair of objects of $\CC$, $(T_{a}, T_{b})$) such that
\[ (T_{a}, T_{b}) \cong F(T_{a}, T_{b}) . \]
This situation has been extensively studied in Category Theory and Theoretical Computer Science \cite{Rut,Barr,AJ}. In particular, the notion of \emph{final coalgebra} provides a canonical form of solution. Once again, previous work has focussed on `one-person' situations, although the tools needed for two- or multi-agent forms of solution --- essentially the ability to solve simultaneous equations --- are already in hand.
We shall briefly review some standard notions on coalgebra before applying them to the construction of assumption-complete models.

\subsection{Brief Review of Coalgebra}
Let $F : \CC \rarr \CC$ be a functor.
An \emph{$F$-coalgebra} is a pair $(A, \alpha)$ where $A$ is an object of $\CC$, and $\alpha$ is an  arrow $\alpha : A \rarr FA$. We say that $A$ is the \emph{carrier} of the coalgebra, while $\alpha$ is the \emph{behaviour map}.

An \emph{$F$-coalgebra homomorphism} from $(A, \alpha)$ to $(B, \beta)$ is an arrow $h : A \rarr B$ such that
\begin{diagram}
A & \rTo^{\alpha} & FA \\
\dTo^{h} & & \dTo_{Fh} \\
B & \rTo_{\beta} & FB
\end{diagram}
$F$-coalgebras and their homomorphisms form a category $\FCoalg$.

An $F$-coalgebra $(C, \gamma)$ is \emph{final} if for every $F$-coalgebra $(A, \alpha)$ there is a unique homomorphism from $(A, \alpha)$ to $(C, \gamma)$, \ie if it is the terminal object in $\FCoalg$.

\begin{proposition}
If a final $F$-coalgebra exists, it is unique up to isomorphism.
\end{proposition}

\begin{proposition}[Lambek Lemma]
If $\gamma : C \rarr FC$ is final, it is an isomorphism
\end{proposition}

Final coalgebras subsume what are known as \emph{terminal models} in the literature on type spaces. A standard way of constructing final coalgebras is by a `terminal sequence' \cite{Barr,Worrell}
\[ \Term \leftarrow F(\Term) \leftarrow F^{2}(\Term) \leftarrow \cdots F^{k}(\Term) \leftarrow \cdots \]
This sequence continues at a limit ordinal $\lambda$ by taking the limit of the diagram constructed at the previous stages, and at a successor ordinal $\lambda + 1$ by applying $F$ to all arrows $A_{\mu} \leftarrow A_{\nu}$ constructed at previous stages, with $\mu \leq \nu$. If at some  stage $\lambda$ the arrow $A_{\lambda} \leftarrow A_{\lambda + 1 } = F(A_{\lambda})$ is an isomorphism, we have constructed the final coalgebra.\footnote{This method will work for any monic-preserving accessible endofunctor on a locally  presentable category \cite{Worrell}.}
This form of construction is used in the literature on \emph{Harsanyi type spaces} \cite{Harsanyi} to construct what are known as \emph{universal models}. Aviad Heifetz and Dov Samet gave the first construction of a universal type space  in the category of measurable spaces \cite{HS}, following other work in more restricted contexts. Subsequently,  Larry Moss and Ignacio Viglizzo made explicit use of final coalgebra ideas to clarify and generalize this construction  \cite{MV}. Their focus was on the category of measurable spaces. Our contribution here is to set the discussion in a wider context, emphasizing the construction of interactive belief models which are assumption complete.

Thus these well-developed methods from Theoretical Computer Science can be used to address the following question raised by Brandenburger and Keisler:
\begin{quotation}
We end by noting that, to the best of our knowledge, no general treatment exists of the relationship between universal, complete, and terminal
models (absent specific structure). Such a treatment would be very useful.
\end{quotation}

\noindent The topic deserves a fuller treatment than is possible here. We shall content ourselves with giving some examples where known results on the existence of final coalgebras can be applied to yield assumption complete models.

\subsection{Application to Assumption Complete Models}
We begin by noting that standard results allow us to lift one-person to two- (or multi-)agent constructions.
Suppose we have endofunctors $G_{1}, G_{2} : \CC \rarr \CC$. We can define a functor
\[ G : \CC \times \CC \rarr \CC \times \CC \; :: \; G(X, Y) = (G_{1}(Y), G_{2}(X)) . \]
Note that this directly generalizes our definition of $F$ from $F_{a}$ and $F_{b}$. 
We have $G = (G_{1} \times G_{2}) \circ \mathsf{twist}$. It is standard that if $G_{1}$ and $G_{2}$ satisfy continuity or accessibility hypotheses which guarantee that they have final coalgebras, so will $G$. 

Note that the final sequence for $G$ will have the form
\[ (\Term, \Term) \leftarrow (G_{1}(\Term), G_{2}(\Term)) \leftarrow (G_{1}(G_{2}(\Term)), G_{2}(G_{1}(\Term)) \leftarrow  \]
\[ \cdots \leftarrow ((G_{1}\circ G_{2})^{k}(\Term), (G_{2}\circ G_{1})^{k}(\Term)) \leftarrow \cdots \]
This `symmetric feedback' is directly analogous to constructions which arise in Geometry of Interaction and the Int construction \cite{NFGOI,Retracing,CG}. It is suggestive of a compositional structure for interactive belief models.

We shall now consider three specific settings where the general machinery we have described can be applied to construct assumption complete models as final coalgebras. In each case we must specify the ambient category $\CC$, and the functor $\PP$.

\subsubsection{Sets}
We firstly consider $\Set$, the category of sets and functions. Our candidate for $\PP$ is a variant of the powerset functor. We take $\PP(X) = \pkap(X)$, the collection of all subsets of $X$ of cardinality less than $\kappa$, where $\kappa$ is an inaccessible cardinal.\footnote{Alternatively, and essentially equivalently, we can follow Peter Aczel \cite{AczelM89}, and work over the (`superlarge') category of classes, taking $\PP(A)$ to be the class of sub-\emph{sets} of a class $A$.}
It is standard that, for any sets $S_a$, $S_b$,  the functors $F_a$ and $F_b$ are accessible, and hence so is the functor $F = (F_a \times F_b) \circ \mathsf{twist} : \CC \times \CC \rarr \CC \times \CC$. Hence we get a final coalgebra
\[ \gamma : (T_a, T_b) \riso (\pkap(S_a \times T_b), \pkap(S_b \times T_a)) . \]
This yields an assumption complete belief model, as previously discussed.

Note that the terminal sequence for this functor is always transfinite, as analyzed in detail in \cite{Worrell}. Even in the case $\kappa = \omega$ (finite subsets), $\omega + \omega$ stages are required for convergence to the final coalgebra.

\subsubsection{Stone Spaces}

Another convenient setting for final coalgebra is the category of \emph{Stone spaces}, \ie totally disconnected compact Haussdorff spaces \cite{cookstour,KupkeKV04}. By Stone duality, this category is dual to the category of Boolean algebras. Our candidate for $\PP$ here is the \emph{Vietoris powerspace construction} \cite{Michael}.
In \cite{cookstour}, one can find essentially a treatment of the one-person case of the situation being considered here. The final coalgebra constructed here is closely related to the model built in a more concrete fashion in \cite{BK}. We get stronger properties (isomorphism rather than surjection) and a clearer relation to general theory.

In this case, the final coalgebra is reached after $\omega$ stages of the terminal sequence, because of continuity properties of the functor.

\subsubsection{Algebraic Lattices}
As a final example, we venture into the realm of Domain theory \cite{CLD,AJ}. We work in the category of  algebraic lattices and Scott-continuous maps (those preserving directed joins). We have two convenient choices for $\PP$: the \emph{lower} and \emph{upper powerdomain} constructions, both well-studied in Domain theory. In the first case, we take the lattice of Scott-closed subsets of an algebraic lattice, ordered by inclusion. In the second, we take the subsets which are compact in the Scott topology, and upwards closed in the partial ordering, ordered by reverse inclusion. In either case, we obtain a continuous functor, which converges to the final coalgebra in $\omega$ stages of the terminal sequence.

\subsubsection{Closure under logical constructions}
We have constructed models which are assumption complete in a semantic sense, with respect to the predicates specified by the functor $\PP$. A further issue is how expressive these collections of predicates are; this can be made precise in terms of which logical constructions they are closed under, and hence which logics can be interpreted. Brandenburger and Keisler show that their topological belief model is closed under conjunction, disjunction, existential and universal quantification, and constructions corresponding to the \textbf{assumes} and \textbf{believes} modalities. The same  arguments show that our model in Stone spaces is closed under these constructions. Similar arguments show that the model in $\Set$ is also closed under these constructions. In this case, closure under the \textbf{believes} modality requires that if a set $S$ has cardinality less than $\kappa$, so does its powerset. This follows from the inaccessibility of $\kappa$. Finally, the models in algebraic lattices are also closed under these constructions, with the proviso that appropriate order-theoretic saturation (upwards or downwards closure) must be applied in some cases.

These models also allow for various forms of recursive definition. We leave a detailed account to an extended version of this paper.

\section{Further Directions}

There are a number of natural directions to be pursued. One is to a more comprehensive account of the construction of belief models and type spaces, taking full advantage of the use of categorical methods, and of developments in coalgebraic logic. Another is to a finer analysis of the use of completeness hypotheses in justifying solution concepts for games. Finally, we would like to pursue the broader agenda of understanding the mathematical structure of interaction, and the scope of interactive versions of logical and mathematical phenomena which have previously only been studied in `one-person' versions.

\section{Acknowledgements}
This research was supported by the EPSRC grant  EP/F067607/1 and by ONR.

\bibliographystyle{plain}

\bibliography{biblio}

\begin{thebibliography}{10}

\bibitem{Retracing}
S.~Abramsky.
\newblock Retracing some paths in process algebra.
\newblock In U.~Montanari and V.~Sassone, editors, {\em {CONCUR '96}:
  Concurrency Theory, 7th International Conference}, pages 1--17.
  Springer-Verlag, 1996.

\bibitem{cookstour}
S.~Abramsky.
\newblock A {C}ook's tour of the finitary non-well-founded sets.
\newblock In Sergei Artemov, Howard Barringer, Artur d'Avila Garcez, Luis~C.
  Lamb, and John Woods, editors, {\em We Will Show Them: Essays in honour of
  Dov Gabbay}, volume~1, pages 1--18. College Publications, 2005.

\bibitem{NFGOI}
S.~Abramsky and R.~Jagadeesan.
\newblock New foundations for the geometry of interaction.
\newblock In {\em Information and Computation, 111(1)}, pages 53--119, 1994.

\bibitem{AJ}
S.~Abramsky and A.~Jung.
\newblock Domain theory.
\newblock In S.~Abramsky, D.~Gabbay, and T.~S.~E. Maibaum, editors, {\em
  Handbook of Logic in Computer Science}, pages 1--168. Oxford University
  Press, 1994.

\bibitem{CG}
S.~Abramsky and P.-A. Melli{\'e}s.
\newblock Concurrent games and full completeness.
\newblock In {\em Proceedings of the Fourteenth International Symposium on
  Logic in Computer Science}, pages 431--442. Computer Society Press of the
  {IEEE}, 1999.

\bibitem{AczelM89}
Peter Aczel and Nax~Paul Mendler.
\newblock A final coalgebra theorem.
\newblock In David~H. Pitt, David~E. Rydeheard, Peter Dybjer, Andrew~M. Pitts,
  and Axel Poign{\'e}, editors, {\em Category Theory and Computer Science},
  volume 389 of {\em Lecture Notes in Computer Science}, pages 357--365.
  Springer, 1989.

\bibitem{Barr}
Michael Barr.
\newblock Terminal coalgebras in well-founded set theory.
\newblock {\em Theor. Comput. Sci.}, 114(2):299--315, 1993.

\bibitem{BS}
P.~Battigalli and M.~Siniscalchi.
\newblock Strong belief and forward-induction reasoning.
\newblock {\em Journal of Economic Theory}, 106:356--391, 2002.

\bibitem{BFK}
A.~Brandenburger, A.~Friedenberg, and H.J. Keisler.
\newblock Admissibility in games.
\newblock {\em Econometrica}, 76:307--352, 2008.

\bibitem{BK}
Adam Brandenburger and H.~Jerome Keisler.
\newblock An impossibility theorem on beliefs in games.
\newblock {\em Studia Logica}, 84(2):211--240, November 2006.

\bibitem{Butz}
Carsten Butz.
\newblock Regular categories and regular logic.
\newblock Technical Report LS-98-2, {BRICS}, October 1998.

\bibitem{CLD}
G.~Gierz, K.~H. Hofmann, K.~Keimel, J.~D. Lawson, M.~Mislove, and D.~S. Scott.
\newblock {\em Continuous Lattices and Domains}.
\newblock Number~93 in Encyclopedia of Mathematics and its Applications.
  Cambridge University Press, 2003.

\bibitem{Harsanyi}
John~C. Harsanyi.
\newblock Games with incomplete information played by "{B}ayesian" players,
  {I--III}. {P}art {I}. {T}he basic model.
\newblock {\em Management Science}, 14(3), 1967.

\bibitem{HS}
A.~Heifetz and D.~Samet.
\newblock Topology-free typology of beliefs.
\newblock {\em Journal of Economic Theory}, 82:324--381, 1998.

\bibitem{KupkeKV04}
Clemens Kupke, Alexander Kurz, and Yde Venema.
\newblock Stone coalgebras.
\newblock {\em Theor. Comput. Sci.}, 327(1-2):109--134, 2004.

\bibitem{Law}
F.~William Lawvere.
\newblock Diagonal arguments and cartesian closed categories.
\newblock {\em Lecture Notes in Mathematics}, 92:134--145, 1969.

\bibitem{Michael}
E.~Michael.
\newblock Topologies on spaces of subsets.
\newblock {\em Trans. Amer. Math. Soc.}, 71:152--182, 1951.

\bibitem{MV}
Lawrence~S. Moss and Ignacio~D. Viglizzo.
\newblock Final coalgebras for functors on measurable spaces.
\newblock {\em Inf. Comput.}, 204(4):610--636, 2006.

\bibitem{Pacuit07}
Eric Pacuit.
\newblock Understanding the {B}randenburger-{K}eisler paradox.
\newblock {\em Studia Logica}, 86(3):435--454, 2007.

\bibitem{Pierce}
Benjamin~C. Pierce.
\newblock {\em Basic Category Theory for Computer Scientists}.
\newblock MIT Press, 1991.

\bibitem{Rut}
Jan J. M.~M. Rutten.
\newblock Universal coalgebra: a theory of systems.
\newblock {\em Theor. Comput. Sci.}, 249(1):3--80, 2000.

\bibitem{SV}
J.~Soto-Andrade and F.~J. Varela.
\newblock Self-reference and fixed points: a discussion and an extension of
  {L}awvere's theorem.
\newblock {\em Acta Applicandae Mathematicae}, 2:1--19, 1984.

\bibitem{JvO}
Jaap van Oosten.
\newblock Basic category theory.
\newblock Technical Report LS-95-1, {BRICS}, January 1995.

\bibitem{Worrell}
James Worrell.
\newblock Terminal sequences for accessible endofunctors.
\newblock {\em Electr. Notes Theor. Comput. Sci.}, 19, 1999.

\bibitem{Yan}
Noson~S. Yanofsky.
\newblock A universal approach to self-referential paradoxes and fixed points.
\newblock {\em Bulletin of Symbolic Logic}, 9(3):362--386, 2003.

\end{thebibliography}

\end{document}